\documentclass[12pt]{amsart}
\usepackage{geometry}               
\usepackage{graphicx}
\usepackage{amssymb}
\usepackage{url}
\usepackage{epstopdf}

\usepackage{lineno}

\newtheorem{theorem}{Theorem}[section]
\theoremstyle{plain}

\newtheorem{corollary}[theorem]{Corollary}
\newtheorem{definition}[theorem]{Definition}

\newtheorem{lemma}[theorem]{Lemma}

\newtheorem{prop}[theorem]{Proposition}
\newtheorem{remark}[theorem]{Remark}

\def\Rn{\mathbb{R}^n}
\def\R{\mathbb{R}}
\def\Tg{T_\gamma}
\def\Tgm{T^{mod}_\gamma}
\def\Lg{\Lambda_\gamma}
\def\W{W^{1,2}_0(K,\gamma)}
\def\dK{\partial K}
\def\dg{d\gamma}

\def\Kt{K_{\tb}}
\def\Gt{\partial\Kt}

\def\tb{t}

\def\tt{u}
\def\tmax{\tt_{\max}}

\def\uv{t}

\def\nn{\nonumber}

\title[Kohler-Jobin meets Ehrhard]{Kohler-Jobin meets Ehrhard: the sharp lower bound for the Gaussian principal frequency while the Gaussian torsional rigidity is fixed, via rearrangements}
\author[Orli Herscovici, Galyna V. Livshyts]{Orli Herscovici, Galyna V. Livshyts}
\address{School of Mathematics, Georgia Institute of Technology,  Atlanta, GA}
\email{oherscovici3@gatech.edu}
\address{School of Mathematics, Georgia Institute of Technology, Atlanta, GA} \email{glivshyts6@math.gatech.edu}

\begin{document}

\begin{abstract}
In this note, we provide an adaptation of the Kohler-Jobin rearrangement technique to the setting of the Gauss space. As a result, we prove the Gaussian analogue of the Kohler-Jobin resolution of a conjecture of P\'{o}lya-Szeg\"o: when the Gaussian torsional rigidity of a domain is fixed, the Gaussian principal frequency is minimized for the half-space. At the core of this rearrangement technique is the idea of considering a ``modified'' 
torsional rigidity, with respect to a given function, and rearranging its layers to half-spaces, in a particular way; the Rayleigh quotient decreases with this procedure.

We emphasize that the analogy of the Gaussian case with the Lebesgue case is not to be expected here, as in addition to some soft symmetrization ideas, the argument relies on the properties of some special functions; the fact that this analogy does hold is somewhat of a miracle.
\end{abstract}

\maketitle

\section{Introduction}

Consider a log-concave measure $\mu$ on $\R^n$ with density $e^{-V}$, for some differentiable convex function $V:\R^n\rightarrow\R$, and its associated Laplacian
$$L\cdot=\Delta \cdot - \langle \nabla \cdot, \nabla V\rangle,$$ 
where $\Delta =\sum_{j=1}^n\partial^2_{x_j}$. 
An open bounded set $K\subset \R^n$ with Lipschitz continuous boundary is called a domain. Let $W^{1,2}(K,\mu)$ stand for the weighted Sobolev space -- the space of $L^2$ functions with the first order weak derivative in $L^2$,  and $\nabla$ for the weak gradient. Letting $W_0^{1,2}(K,\mu)= W^{1,2}(K,\mu)\cap \{u:u|_{\partial K}=0\}$ to be the closure of $C^\infty_0(K)$ with respect to the norm $||\nabla u||_{L^2(K)}$ (see also  \cite{Brasco2014}), where the boundary value is understood in the sense of the trace (see \cite{Evans-pde}), define the $\mu$-torsional rigidity as 
$$T_{\mu}(K)=\sup_{u\in W_0^{1,2}(K,\mu)} \frac{(\int_K u \,d\mu)^2}{\int_K |\nabla u|^2 d\mu}.$$
See section 2 for the details, or, for example, P\'{o}lya and Szeg\"o \cite{PS1951}. 

The $\mu$-principal frequency of a domain $K$ is defined to be 
$$\Lambda_{\mu}(K):= \inf_{u\in W_0^{1,2}(K,\mu)} \frac {\int_K |\nabla u|^2 d\mu}{\int_K u^2 d\mu}.$$
Note that the torsional rigidity is monotone increasing while the principle frequency is monotone decreasing, i.e. whenever $K\subset M,$ we have $T_{\mu}(K)\leq T_{\mu}(M)$ and $\Lambda_{\mu}(K)\geq \Lambda_{\mu}(M)$. 

In the case when $\mu$ is the Lebesgue measure and $L=\Delta,$ these quantities have been studied extensively, and are intimately tied with the subject of isoperimetric inequalities. See, e.g. Kawohl \cite{kaw}, P\'olya and Szeg\"{o} \cite{PS1951}, 
Burchard \cite{Almut}, Lieb and Loss \cite{LLbook}, Kesavan \cite{Kesavan}, or Vazquez \cite{Vasquez}. In particular, the Faber-Krahn inequality \cite{Faber, Krahn1, Krahn2} states that the (Lebesgue) principal frequency of a domain $K$ of a fixed Lebesgue measure is \emph{minimized} when $K$ is an Euclidean ball. The result of Saint-Venant (see e.g. \cite{PS1951}) states that, conversely, the torsional rigidity of a domain $K$ of a fixed Lebesgue measure is \emph{maximized} when $K$ is an Euclidean ball. 

The easiest way to prove these results is via rearrangements. For a set $K$ in $\R^n$, denote by $K^*$ the centered at the origin Euclidean ball of the same Lebesgue measure as $K$. Recall that the Schwartz rearrangement of a non-negative function $u:K\rightarrow \R$ is the function $u^*:K^*\rightarrow \R$ whose level sets $\{u^*\geq t\}$ are all Euclidean balls centered at the origin, and such that $|\{u\geq t\}|=|\{u^*\geq t\}|$ for all $t\in\R$ (where $|\cdot|$ stands for the Lebesgue measure for sets, or the Euclidean norm for vectors.) The P\'olya-Szeg\"o principle \cite{PS1951} (which is a consequence of the isoperimetric inequality) implies that $\int_K |\nabla u|^2 dx \geq \int_{K^*} |\nabla u^*|^2 dx$, while the definition of the symmetrization yields that $\int_K u^2 dx=\int_{K^*} (u^*)^2 dx$, and, for a non-sign-changing function $u,$ $\int_K u \,dx=\int_{K^*} u^* dx$. Therefore, the Faber-Krahn and the Saint-Venant results follow (together with some additional information that the extremal function for the torsional rigidity is non-negative, as can be deduced via the maximal principle). 

In the case of the Gaussian measure $\gamma$ (which is the measure with density $\frac{1}{(\sqrt{2\pi})^n} e^{-\frac{|x|^2}{2}}$, $x\in\R^n$), the analogue of the Schwartz rearrangement was developed by Ehrhard \cite{Ehr, Ehr2}. The Euclidean balls are replaced with the Gaussian isoperimetric regions, which are nested 
half-spaces (see Sudakov and Tsirelson \cite{ST}, Borell \cite{Bor}). In the Gaussian world, $K^*$ is the half-space of the same Gaussian measure as the domain $K$, and $u^*$ is the function whose level sets are half-spaces, and such that $\gamma(\{u\geq t\})=\gamma(\{u^*\geq t\})$ for all $t\in\R$. 
The P\'olya-Szeg\"o principle is replaced with the analogous 
Ehrhard principle, which yields $\int_K |\nabla u|^2 d\gamma \geq \int_{K^*} |\nabla u^*|^2 d\gamma.$ As a result, whenever the Gaussian measure of the domain $K$ is fixed, the Gaussian principal frequency $\Lambda_{\gamma}(K)$ is minimized when $K$ is a half-space (see Carlen and Kerce \cite{Carlen}), and the Gaussian torsional rigidity $T_{\gamma}(K)$ is maximized when $K$ is a half-space (see e.g. Livshyts \cite{Livshyts2021}).

In the Lebesgue world, P\'olya and Szeg\"o asked another natural question: if for a set $K,$ not its measure, but its torsional rigidity is fixed, then is the principal frequency still minimized on the Euclidean ball? This question was answered in the affirmative by Kohler-Jobin \cite{KJ1982, KJ1978} back in the 1970s, whose rearrangement technique is based on keeping level sets to be of the same torsional rigidity instead of the same Lebesgue measure and relaxing definition of the torsional rigidity. For further generalizations and applications, as well as a nice exposition of the topic, see Brasco \cite{Brasco2014}.

In this paper, we develop the Gaussian analogue of the Kohler-Jobin rearrangement, and show

\begin{theorem}\label{mainThm}
For any open domain $K\subset\Rn$, letting $H$ be the half-space such that $T_{\gamma}(K)=T_{\gamma}(H)$, we have $\Lambda_{\gamma}(K)\geq \Lambda_{\gamma}(H).$
\end{theorem}

We would like to emphasize that the analogy of our work and the works of Kohler-Jobin \cite{KJ1982, KJ1978} and Brasco \cite{Brasco2014} is not to be expected, like in the case of the inequalities of Faber-Krahn and Saint-Venant! Our proof relies heavily on the particular properties of certain special functions, and not just on soft properties of rearrangements (see for example Lemma 2.7 and Theorem 4.1). 

This paper is organized as follows. In Section 2 we discuss some preliminaries. In Section 3 we discuss the Gaussian modified torsional rigidity. In Section 4 we define the Gaussian version of the Kohler-Jobin rearrangement, study its properties, and prove Theorem \ref{mainThm}.

\textbf{Acknowledgement.} 
The authors would like to thank the anonymous referees for their useful comments and suggestions that improved the presentation and quality of this paper.

The authors are supported by the NSF CAREER DMS-1753260. 

\section{Preliminaries}

\subsection{General}

Fix a domain  
$K$ in $\R^n$. 
For a point $x\in\partial K,$ we denote by $n_x$ the unit outer normal vector to the boundary of $K$ at $x$; if $K$ is convex then $n_x$ is defined almost everywhere on the boundary of $K$. 

Recall the co-area formula (see e.g. Brothers and Zeimer \cite{BZ1988}, Lemma 4.1 from Carlen and Kerce \cite{Carlen}, or \cite{coarea}) which states, with the convention of Brothers and Ziemer:
\begin{equation}
\int_{\Rn} |\nabla u|f\, dx=\int_{0}^{\infty}\int_{\{u=s\}}f\,d\mathcal{H}^{n-1} ds,\nn
\end{equation}
where $f\in L_{\infty}(\Rn,\gamma)$, $u\in W^{1,1}(\Rn,\gamma)$ is a nonnegative Borel measurable function, and $\mathcal{H}^{n-1}$ denotes the $n-1$ dimensional Hausdorff measure. Here $W^{1,1}(\Rn,\gamma)$ is a weighted Sobolev space of functions with integrable distributional weak gradient. By replacing $f(x)$ with $\frac{1}{(\sqrt{2\pi})^n}\,f(x)e^{-\frac{|x|^2}{2}}$, we get 
\begin{equation}\label{coarea}
\int_{\Rn} |\nabla u|f\, d\gamma=\int_{0}^{\infty}\int_{\{u=s\}}f\,d\gamma_{\partial\{u=s\}} ds,
\end{equation}
where by $d\gamma_{\partial M}$, for a surface $M,$ we denote the measure $\frac{1}{(\sqrt{2\pi})^n}\,e^{-\frac{|x|^2}{2}}\,d \mathcal{H}^{n-1}$, i.e. the boundary measure with the Gaussian weight.

Recall the Ornstein-Uhlenbeck operator \cite{bogachev} defined as 
\begin{equation}\label{OUO}
Lu=\Delta u-\langle x,\nabla u\rangle,\end{equation}
which verifies the integration by parts identity for the standard Gaussian measure:
$$\int_{K} vLu\, d\gamma=-\int_K \langle \nabla u,\nabla v\rangle d\gamma+\int_{\partial K} v\langle \nabla u,n_x\rangle
d\gamma_{\partial K},$$
 for Lipschitz continuous functions $u$, $v$, and a domain $K$. In fact, the ``integration by parts'' identity also serves as the definition of $L$ on Sobolev spaces. Recall also the following classical existence result, see e.g. \cite{Livshyts2021}.

\begin{theorem}[Dirichlet boundary condition]\label{exist-dir}
Let $K$ be a Lipschitz domain. Let $F\in L^2(K,\gamma)$. Then there exists a unique function $u\in W^{1,2}(K,\gamma)$ which is a solution of 
$$
\begin{cases}  Lu=F &\, on \,\,\,K,\\ u=0 &\, on \,\,\,\partial K.\end{cases}$$

Moreover, if $\partial K$ is $C^{2}$ then $u\in C^{2}(int(K))\cap C^1(\bar{K}).$

Furthermore, if $F\leq 0$ then $u\geq 0,$ and the level sets of $u$ are nested.
\end{theorem}

Note that the ``furthermore'' part of the theorem follows from the maximal principle.

\subsection{Equivalent definitions of the Gaussian torsional rigidity}

\begin{definition} For a  
domain $K\subset\Rn$, define the Gaussian 
torsional rigidity by
\begin{equation}
\Tg(K)=\sup_{v\in\W
}\frac{(\int_K v\,\dg)^2}{\int_K|\nabla v|^2 \dg}.\label{TDef1}
\end{equation}
\end{definition}
This object (with a different normalization!) was considered in \cite{Livshyts2021}. Let us also state

\begin{definition} For a  
domain $K\subset\Rn$, define the Gaussian 
torsion function of $K$ to be the unique function $v\in W^{1,2}(K)$ satisfying the differential equation
$$
\begin{cases}  Lv=-1 &\, on \,\,\,K,\\ v=0 &\, on \,\,\,\partial K.\end{cases}$$
\end{definition}

We note that by the maximal principle, the Gaussian torsion function is necessarily non-negative.

The Gaussian torsional rigidity admits several equivalent definitions.

\begin{prop} \label{prop1-3} 
The following are equivalent:
\begin{enumerate}
\item $\Tg(K)=\inf_{u\in W^{1,2}(K,\gamma): \,\,Lu=-1} \int_K |\nabla u|^2 d\gamma;$
\item $\Tg(K)$ is the Gaussian torsional rigidity, i.e. $\Tg(K)=\sup_{\substack{v\in\W}}\frac{(\int_K v\,\dg)^2}{\int_K|\nabla v|^2 \dg};$
\item $\Tg(K)= \sup_{\substack{v\in\W}} \left({-}\int_K |\nabla v|^2\dg\label{TDef2}{+2}\int_K v\,\dg\right);$
\item $T_{\gamma}(K)=\int_K v d\gamma=\int_K |\nabla v|^2 d\gamma= {-}\int_K |\nabla v|^2\dg\label{TDef2}{+2}\int_K v\,\dg,$
where $v$ is the Gaussian torsion function of $K.$
\end{enumerate}
\end{prop}
\begin{proof} For any $v\in \W$  and any $u\in W^{1,2}(K,\gamma)$ with $Lu=-1$, we have, by Cauchy's inequality,

$$\int_K |\nabla u|^2 d\gamma\geq \frac{\left(\int_K \langle \nabla u,\nabla v\rangle d\gamma\right)^2}{\int_K |\nabla v|^2 d\gamma}=\frac{\left(\int_K v\,\dg\right)^2}{\int_K |\nabla v |^2 d\gamma},$$
where for the last identity we used the definition of the operator $L$ on Sobolev spaces (see, for example, the book of Evans and Gariepy \cite{Evans-Gariepy}): we say that $Lu=F$ for a Sobolev function $u$ if for any smooth test-function $v$ with $v=0$ on $\partial K$ we have $\int_K vF d\gamma=-\int_K \langle \nabla v,\nabla v\rangle d\gamma$, and the more general fact (as used above) for Sobolev functions $v$ follows by a standard approximation argument. Therefore,
$$ 
\inf_{\substack{u\in W^{1,2}(K,\gamma):\\ Lu=-1}} \int_K |\nabla u|^2 d\gamma
\geq
\sup_{\substack{v\in\W}}\frac{(\int_K v\,\dg)^2}{\int_K|\nabla v|^2 \dg}.
$$

In order to see that the equality is attained, consider (on both sides) the function $v\in\W$ satisfying the differential 
equation
$$
\begin{cases}  Lv=-1 &\, on \,\,\,K,\\ v=0 &\, on \,\,\,\partial K.\end{cases}$$
Thus (1) and (2) are equivalent. 

In order to see that (1) and (3) are equivalent, use the  inequality $$\int_K |\nabla u|^2 d\gamma\geq \int_K \left(2\langle \nabla u,\nabla v\rangle -|\nabla v|^2 \right)d\gamma,$$ and the fact that for any $v\in \W$ and any $u\in W^{1,2}(K,\gamma)$ with $Lu=-1$, we have that $\int_K \langle \nabla u,\nabla v\rangle\,d\gamma=\int_K v\,d\gamma$,
and argue in the same manner as before.

The equivalence of (4) to the rest of the definitions follows automatically by the choice of the extremizing function. Note that with this choice of $v,$ we have, integrating by parts, $\int_K v \,d\gamma=\int_K |\nabla v|^2 d\gamma.$
\end{proof}

We shall record the following classical fact:

\begin{lemma}[torsional rigidity is monotone]\label{T-mon}
For a pair of domains $K$ and $M$, $T_{\gamma}(K)\geq T_{\gamma}(M)$ whenever $M\subset K.$
\end{lemma}
\begin{proof} Let $v$ be the torsion function on $M$ (i.e. the function $v$ maximizing the righthand side of the \eqref{TDef1}), and extend it by zero to the whole of $K$; let us call the resulting function $\tilde{v}.$ Note that the result is still a Sobolev function since $v|_{\partial M}=0.$ We have
$$\Tg(M)=\frac{(\int_M v\dg)^2}{\int_M|\nabla v|^2 \dg}=\frac{(\int_K \tilde{v}\dg)^2}{\int_K|\nabla \tilde{v}|^2 \dg}\leq \Tg(K),$$
where in the last line, the definition of the torsional rigidity was used again.
\end{proof}

\medskip
\medskip

Let us also recall the Gaussian analogue of the Saint-Venant theorem (see e.g. Proposition 5.6 in Livshyts \cite{Livshyts2021}):

\begin{lemma}\label{SV}
Of all domains $K$ with Gaussian measure $a\in [0,1]$, the torsional rigidity is maximized for the half-space $\{x_1\leq \Phi^{-1}(a)\}.$
\end{lemma}
Here the function $\Phi^{-1}(t)$ is the inverse function of the standard Gaussian distribution $\Phi(t)=\int\limits_{-\infty}^{t}\frac{1}{\sqrt{2\pi}}\exp(-\frac{s^2}{2})\,ds$.
\begin{remark} The proposition 5.6 in Livshyts \cite{Livshyts2021} was proved for convex set $K$, but the same arguments can be applied also for all the measurable sets. We do not bring this proof here.
\end{remark}

\subsection{Gaussian torsional rigidity of a half-space}

Let us denote the right half-space
$$H_s=\{x=(x_1,\ldots,x_n)\in\R^n\,\big|\,x_1\geq s\}.$$
Define the function $T:\R\rightarrow \R$ to be the Gaussian torsional rigidity of the half-space:
$$T(s)=T_{\gamma}(H_{s}).$$
We shall show
\begin{lemma}\label{torhs} $T'(s)=-\sqrt{2\pi} \,e^{\frac{s^2}{2}}\gamma(H_s)^2.$
\end{lemma}
\begin{proof} Note that the torsion function $v_s: H_{s}\rightarrow\R$ is given by 
$$v_s(x)=V(x_1)-V(s),$$
where $V(t)$ is such a function that
$$V'(t)=e^{\frac{t^2}{2}}\int_{t}^{\infty} e^{-\frac{\tau^2}{2}} d\tau.$$
Indeed, $v_s\in W^{1,2}(H_s)$, and one may check directly that $Lv_s=-1$, and clearly $v_s(x)=0$ when $x\in\partial H_{s}.$ Therefore, by (4) of Proposition \ref{prop1-3}, 
$$T(s)=T_{\gamma}(H_{s})=\int_{H_{s}} |\nabla v_s|^2 d\gamma=\frac{1}{\sqrt{2\pi}}\int_{s}^{\infty} V'(t)^2 \,e^{-\frac{t^2}{2}} dt,$$
where in the last passage we used Fubini's theorem. Thus 
$$T'(s)= -\frac{1}{\sqrt{2\pi}}\, V'(s)^2\, e^{-\frac{s^2}{2}}= -\sqrt{2\pi} \,e^{\frac{s^2}{2}}\left(\frac{1}{\sqrt{2\pi}}\int_{s}^{\infty} e^{-\frac{\tau^2}{2}} d\tau\right)^2,$$
	which finishes the proof.
\end{proof}

\subsection{Equivalent definitions of the Gaussian principal frequency}
\begin{definition} For a  
domain $K\subset\Rn$, define the Gaussian 
principal frequency $\Lg(K)$ by
\begin{equation*}
\Lg(K)=\inf_{\substack{v\in\W}}\frac{\int_K |\nabla v|^2\dg}{\int_K v^2 \dg}.
\label{LamDef1}
\end{equation*}
\end{definition}
We outline the following classical fact (see e.g. P\'olya and Szeg\"o for the Lebesgue version \cite{PS1951}.)
\begin{prop} \label{Prop3-2}
$\Lg(K)$ is the first nonzero eigenvalue of $L$ with the Dirichlet boundary condition on a domain $K$.
\end{prop}
\begin{proof}
We aim to show that $\Lg(K)$ is the smallest positive number $\Lambda_{\gamma}$ such that there exists a function $u\in W_0^{1,2}(K)$ such that $u \not\equiv0$ with 
\begin{equation}\label{LamDef2}
Lu=-\Lambda_{\gamma} u.
\end{equation}
Let $v\in W_0^{1,2}(K)$ be a nonnegative function satisfying the differential equation 
\eqref{LamDef2} in the interior of the domain $K$ and $v|_{\dK}=0$. 

First of all we assume that $f=qv$ for a function $q\in C^2(\Rn)$ defined in $K\cup\dK$. Then
\begin{align*}
\int_K\left(|\nabla f|^2-\Lg f^2\right)\dg&=\int_K\left(\left<\nabla(qv),\nabla(qv)\right>-\Lg q^2v^2\right)\dg\\
&=\int_K\left(q^2|\nabla v|^2 +2vq\left<\nabla q,\nabla v\right>+v^2|\nabla q|^2 + q^2vLv\right)\dg\\
&=\int_K|\nabla q|^2 v^2\dg + I_0,
\end{align*}
where 
\begin{equation}
I_0=\int_K\left(q^2|\nabla v|^2+2vq\left<\nabla q,\nabla v\right>+q^2vLv\right)\dg\label{F}
\end{equation}
Integrating by parts and following the definition \eqref{OUO}, we get
\begin{align*}
\int_K 2vq\left<\nabla q,\nabla v\right>\dg&=-\frac{1}{2}\int_K \left<\nabla q^2,\nabla v^2\right>\dg\\
&=-\frac{1}{2}\int_K q^2Lv^2\dg\\
&=-\frac{1}{2}\int_K q^2(2vLv+2|\nabla v|^2)\dg.
\end{align*}
Substituting it into \eqref{F}, we obtain
\begin{equation*}
I_0=\int\left(q^2|\nabla v|^2 -q^2vLv-q^2|\nabla v|^2+q^2vLv\right)\dg=0,
\end{equation*}
which implies
\begin{equation*}
\int|\nabla f|^2\dg=\Lg\int f^2\dg+\int|\nabla q|^2v^2\dg,
\end{equation*}
and therefore
\begin{equation*}
\frac{\int|\nabla f|^2\dg}{\int f^2\dg}\geq \Lg.
\end{equation*}
The equality is attained if $q$ is a constant function, and the statement follows by approximation and mollification \cite{Evans-pde}.
\end{proof}

We shall record the following classical fact.

\begin{lemma}[principal frequency is monotone decreasing]\label{L-mon}
For a pair of domains $K$ and $M$, $\Lambda_{\gamma}(K)\geq \Lambda_{\gamma}(M)$ whenever $K\subset M.$
\end{lemma}
The proof of this Lemma is completely analogous to the proof of Lemma \ref{T-mon}, except we consider infimum in place of supremum, and argue that infimum over a larger class is smaller than the infimum over a smaller class. We leave the details to the reader.

\section{Modified torsional rigidity}

Let us use the notation
\begin{equation}\label{thefunctional}
\Tg[v]=-\int_K|\nabla v|^2\dg+2\int_K v\,\dg,\quad v\in\W.
\end{equation}

For a bounded non-negative $\tt(x)\in W_0^{1,2}(K,\gamma)$ which vanishes at the boundary, let $\tmax$ be the maximum value of $u$ on K. Let us define
\begin{align}
\Kt&=\{x\in K\,|\, \tt(x)>\tb\,\},\nn\\
\Gt&=\{x\in K\,|\, \tt(x)=\tb\,\},\nn 
\end{align}
and
\begin{align}
\gamma(\Kt)&=\int_{\Kt}\dg,\nn\\
\ell(\tb)&=\int_{\Gt}|\nabla \tt|\,d\gamma_{\partial K_t}. 
\label{eqell}
\end{align}

We shall consider a class of \emph{reference functions} $\tt(x)\in W_0^{1,2}(K,\gamma)\cap L_{\infty}(K)$ which are non-negative on $K$ and vanish at the boundary of $K,$ and such that $\frac{\gamma(\Kt)^2}{\ell(\tb)}$ is bounded from above by an absolute constant that does not depend on $t.$

\begin{definition} Consider a reference function $u$. Define $Cl(u)$ to be the class of functions $v(x)\in\W\cap L_{\infty}(K)$ 
which can be presented as $ v=\varphi\circ u(x)$, 
for some real-valued function $\varphi(x)\geq 0,$ defined on the interval $[0, \tmax]$ and satisfying  
$\varphi(0)=0$.

The \emph{Gaussian modified torsional rigidity} $\Tgm(K;\tt(x))$ with respect to this reference function $u$ is given by
\begin{equation}\label{modTR}
\Tgm(K;\tt)=\sup_{v\in Cl(u)}\Tg[v].
\end{equation}
\end{definition}

Note that $\Tg[v]=\Tg[\varphi\circ\tt(x)]$, and thus $\Tgm(K;\tt),$ depends only on the collection of the level sets of $\tt(x)$. It was shown in the Proposition~\ref{prop1-3} that 
\begin{equation*}
\Tg(K)=\sup_{v\in\W}\Tg[v],
\end{equation*}
and therefore, since $Cl(u)\subset W_0^{1,2}(K,\gamma),$ we have
\begin{equation}\label{TRgeqTmod}
\Tg(K)\geq \Tgm(K;\tt)
\end{equation}
for any reference function $u$. Furthermore, this inequality is equality when $u$ is the torsion function. It is important to note that the torsion function $\tt(x)\in W_0^{1,2}(K,\gamma)$ is indeed a reference function: it is non-negative, bounded (by maximum principle), vanishes at the boundary, and $\frac{\gamma(\Kt)^2}{\ell(\tb)}\leq 1.$ Indeed, when $Lu=-1,$ we integrate by parts to see that
$$\gamma(K_t)=-\int_{K_t} Lu d\gamma=\int_{\partial K_t} \langle \nabla u,n_x\rangle d\gamma_{\partial K}=\int_{\partial K_t} |\nabla u| d\gamma_{\partial K} =\ell(t),$$
since the outer unit normal to a level set of a function is collinear to the gradient. Therefore, in the case of the torsion function, we have
$$\frac{\gamma(\Kt)^2}{\ell(\tb)}=\gamma(K_t)\leq 1.$$

We shall show the following analogue of the \cite[Lemma~1]{KJ1982}, \cite[Lemma~1]{KJ1978} (see also a wonderful exposition by Brasco \cite{Brasco2014}), in which we are going to find exactly the maximizing function for the modified torsional rigidity, given an arbitrary reference function $\tt(x)\in W_0^{1,2}(K,\gamma)$.

\begin{prop}\label{FR-Lemma1}
Fix a reference function $u$. The Gaussian modified torsional rigidity of $K$ with respect to $u$ is given by
\begin{equation*}
\Tgm(K;\tt(x))=\int_0^{\tmax} 
\frac{\gamma(\Kt)^2}{\ell(\tb)}\,d\tb.
\end{equation*}
\end{prop}
\begin{proof} Note that by our assumptions on a reference function, the integral $\int_0^{\tmax} 
\frac{\gamma(\Kt)^2}{\ell(\tb)}\,d\tb$ is bounded. Suppose $v=\varphi\circ u,$ for some function $\varphi:\R\rightarrow\R^+$ such that $\varphi(0)=0.$ Suppose without loss of generality that $\varphi\in C^2(\R)$ (otherwise we may approximate it by a smooth function.)

Consider the first summand from the functional $\Tg[v]$ defined in \eqref{thefunctional}:
\begin{equation}\label{above}
\int_K|\nabla v|^2\dg=\int_K|\nabla \varphi(u(x))|^2\dg=\int_K (\varphi')^2|\nabla u|^2\dg.
\end{equation}
By (\ref{above}) and by the co-area formula (\ref{coarea}),
\begin{equation}
\int_K|\nabla v|^2\dg
=\int_0^{\tmax}\left|\frac{d\varphi}{d\uv}\right|^2\int_{\Gt}|\nabla u|\,d\gamma_{\partial K_t}
\,d\uv
=\int_0^{\tmax}\left|\frac{d\varphi}{d\uv}\right|^2\ell({\tb})\,d\uv.\label{4-18}
\end{equation}
If $|\nabla u|>0$ almost everywhere, we could use the co-area formula (\ref{coarea}) to get
$$
\frac{d\gamma(\Kt)}{d\tb}=
\frac{d}{d\tb}\left(\int_{\tb}^\infty\int_{\partial K_s}\frac{1}{|\nabla \tt|}\, 
d\gamma_{\partial K_s} 
ds\right)=-\int_{\Gt}\frac{1}{|\nabla \tt|}\,d\gamma_{\partial K_t}.
\label{eqda}
$$
However, a reference function may not satisfy $|\nabla u|>0$ almost everywhere (also, it may exhibit a variety of bad behaviors, see e.g. Almgren, Leib \cite{AL}). Fortunately, in this case it was shown by Carlen, Kerce  \cite[page 16]{Carlen} that
\begin{equation}\label{star}
\frac{d\gamma(\Kt)}{d\tb}=\frac{d}{d\tb}\left(\int_{\Kt}\frac{1}{|\nabla \tt|}|\nabla \tt|\,\dg\right)\leq -\int_{\Gt}\frac{1}{|\nabla \tt|}\,d\gamma_{\partial K_t}.
\end{equation}
Using the same reasoning (that is, combining Lemma 4.2 from Carlen, Kerce \cite{Carlen} with the co-area formula), we also see, using the fact that $v\geq 0:$
\begin{equation}\label{v}
\int\limits_K v(x)\dg=\int\limits_K\varphi(u(x))\dg=\int\limits_{K\cap\{|\nabla u|\neq 0\}}\frac{\varphi(u(x))}{|\nabla u|}\,|\nabla u|\,\dg+
\int\limits_{K\cap\{|\nabla u|=0\}}\varphi(u(x))\dg
\end{equation}
\begin{equation*}
\leq\int_0^{\tmax}\varphi(\uv)
\int_{\Gt}\frac{1}{|\nabla u|}d\gamma_{\partial K_t} 
d\tb.
\end{equation*}
Combining (\ref{star}) with (\ref{v}), we see
\begin{align}
\int_K v(x)\dg\leq\int_0^{\tmax}\varphi(\uv)\left(-\frac{d\gamma(\Kt)}{d\tb}\right)d\tb= \int_0^{\tmax}\gamma(\Kt)\left(\frac{d\varphi}{dt}\right)
dt,
\label{4-23}
\end{align}
where the last equality follows by the one-dimensional integration by parts, together with the facts that $\gamma(K_{u_{max}})=0$ and $\varphi(0)=0$. 

Combining (\ref{4-18}) and (\ref{4-23}), we estimate the functional $\Tg[v]$ (defined in (\ref{thefunctional})) with $v=\varphi\circ u$ as follows:
\begin{equation*}
\Tg[\varphi(u)]\leq-\int_0^{\tmax} \left|\frac{d\varphi}{d\tb}\right|^2 \ell(\tb) \,d\tb
+2\int_0^{\tmax}\gamma(\Kt)\,\frac{d\varphi}{d\uv}\,d\uv.
\end{equation*}

Since $\ell(\tb)> 0$ for all $0\leq \tb\leq\tmax$ (by our assumptions on a reference function), we can rewrite the last equation as
\begin{align}
\Tg[v]&\leq\int_0^{\tmax}\ell(\tb) \left[2 \,\frac{d\varphi}{d\tb}\frac{\gamma(\Kt)}{\ell(t)}-\left|\frac{d\varphi}{d\tb}\right|^2 \right]d\tb\nn\\
&=\int_0^{\tmax}\ell(\tb) \left[-\left(\frac{d\varphi}{d\tb}-\frac{\gamma(\Kt)}{\ell(t)}\right)^2+\left(\frac{\gamma(\Kt)}{\ell(t)}\right)^2  \right]
d\tb\nn\\
&\leq\int_0^{\tmax}\frac{\gamma^2(\Kt)}{\ell(t)}\,
d\tb.\label{conclud}
\end{align}
Considering the non-negative function $\varphi_0(t)=\int_{0}^{\tb}\frac{\gamma(K_s)}{\ell(s)}\,ds,$ we have $d\varphi_0/d\uv=\gamma(\Kt)/\ell(t)$, and the equality is attained in the last inequality of \eqref{conclud}, while the first inequality turns into equality for a suitable choice of reference functions $u$ such that $\gamma(K\cap\{|\nabla u|=0\})=0$; therefore, 
\begin{equation*}
\Tgm(K;u)=\sup_{\varphi}\, \Tg[\varphi(u)] =\Tg[\varphi_0(u)]=\int_0^{\tmax}\frac{\gamma^2(\Kt)}{\ell(t)}
d\tb,
\end{equation*}
which completes the proof.
\end{proof}

As a corollary, we deduce, using the fact $u(x)-t$ is in the class of reference functions on $K_t$ when $u$ is a reference function on $K,$ and also the fact  that $\ell(t)$ as well as the level sets are the same for $u(x)-t$ and $u(x)$ on $[t, u_{\max}]$:

\begin{corollary}\label{cor4-5}
The functional $\Tgm(\Kt;u(x)-\uv)$ is  the Gaussian modified  torsional rigidity of a level set $\Kt$ with respect to a function $u(x)-\uv$, and it is given by
\begin{equation*}
\Tgm(\Kt;u(x)-\uv)=\int_{\uv}^{\tmax}\frac{\gamma^2(K_s)}{\ell(s)}ds. 
\end{equation*}
\end{corollary}

\section{The Kohler-Jobin-style symmetrization in the Gauss space, and the proof of Theorem \ref{mainThm}}

Let $K$ be a domain in $\R^n$. Let $u\in W_0^{1,2}(K)$ be a non-negative function. 
In this section, we define the analogue of the Kohler-Jobin symmetrization in the Gauss space. The general idea is to correspond the level sets of $u$ to half-spaces whose torsional rigidity is the same as the modified torsional rigidity of the level sets of $u.$ Now, these half-spaces are level sets of rearrangement function $u^{\dagger}$  which will be defined a bit later, but composed with a specific function. Let us describe the details below.

In view of Corollary \ref{cor4-5}, define the distribution function of the modified torsional rigidity of $K$ with respect to $u$ by
\begin{equation}\label{D}
D(t)=\Tgm(\Kt;u(x)-\uv)=\int_{t}^{u_{\max}}\frac{\gamma^2(K_s)}{\ell(s)}ds.
\end{equation}
By Corollary \ref{cor4-5}, 
\begin{equation}\label{D'}
(D^{-1})'(\tau)=\frac{1}{D'(D^{-1}(\tau))}=-\frac {\ell(D^{-1}(\tau))}{\gamma^2(K_{D^{-1}(\tau)})}.
\end{equation}
Note that $D(t)$ is well-defined since $\ell(t)>0$ by our definition of the class of reference functions. Note also that $D$ is differentiable, since it is defined as an antiderivative.

Next, recall from the sub-section 2.3 the notation $H_s=\{x_1\geq s\}$ and 
$$T(s)=T_{\gamma}(H_s).$$ 
Let $K^{\dagger}$ be the half-space of the form $\{x_1\geq s\}$ such that 
$$t_0:=T^{mod}_{\gamma}(K; u)=T_{\gamma}(K^{\dagger}).$$ We shall define the \emph{Ehrhrard-Kohler-Jobin rearrangement} of $u$ to be the non-negative function $u^{\dagger}\in W_0^{1,2}(K^{\dagger})$, non-decreasing in $x_1$, and constant in $x_i$ for $i\neq 1,$ given by 
$$u^{\dagger}(x)=f\circ T(x_1),$$ 
where $f:\R\rightarrow\R$ is the (non-increasing) function given by
\begin{equation}\label{f-def}
\begin{cases}  f'(\tau) =\frac{(D^{-1})'(\tau)\gamma(K_{D^{-1}(\tau)})}{\gamma(H_{T^{-1}(\tau)})},\\ f(t_0)=0. \end{cases}\end{equation}

In other words, $u^{\dagger}=f(\tau)$ on the boundary of the right half-space whose torsional rigidity is $\tau$ 	(as the modified torisional rigidity of the corresponding level set), and so we have, for all $\tau\geq 0,$
\begin{equation}\label{tautol}
T_{\gamma}(K_{D^{-1}(\tau)}) \geq T^{mod}_{\gamma}(K_{D^{-1}(\tau)}; u-D^{-1}(\tau))= T_{\gamma}(H_{T^{-1}(\tau)}).
\end{equation}
Just to clarify, the last identity is a tautology $\tau=\tau$, and the first inequality follows from (\ref{TRgeqTmod}). In view of (\ref{tautol}) together with Lemma \ref{SV} (the Gaussian analogue of Saint-Venant theorem), we conclude that
\begin{equation}\label{levset-meas}
\gamma(K_{D^{-1}(\tau)})\geq \gamma(H_{T^{-1}(\tau)}).	
\end{equation}
Note also that the condition $f(t_0)=0$ means that $u^{\dagger}=0$ on the boundary of $K^{\dagger}$, just as intended.

We are ready to prove

\begin{theorem}\label{mainthm-sym}
The function $u^{\dagger}$ has the following properties:
\begin{enumerate}
\item $\int_K |\nabla u|^2 \,d\gamma=\int_{K^{\dagger}} |\nabla u^{\dagger}|^2 \,d\gamma;$
\item $\int_K u\, d\gamma=\int_{K^{\dagger}}  u^{\dagger} \,d\gamma;$
\item For any convex non-decreasing function $F:\R^+\rightarrow\R^+$ such that $F(0)=0,$ we have $$\int_K F(u) \,d\gamma\leq \int_{K^{\dagger}}  F(u^{\dagger})\, d\gamma.$$
\end{enumerate}
\end{theorem}
\begin{proof} 

\textbf{Part (1).} By the co-area formula (\ref{coarea}), combined with the definition \eqref{eqell} of the function $\ell(s),$ we have
\begin{equation}\label{step1}
\int_K |\nabla u|^2 d\gamma=\int_0^{u_{max}} \ell(s)\,ds.
\end{equation}
Applying change of variables $\tau=D(s),$ with the function $D$ defined in (\ref{D}), we write
\begin{equation}\label{step2}
\int\limits_0^{u_{max}} \ell(s)\, ds= -\int\limits_0^{t_0} (D^{-1})'(\tau)\cdot \ell\circ D^{-1}(\tau) d\tau= \int\limits_0^{t_0} \left((D^{-1})'(\tau)\right)^2\gamma(K_{D^{-1}(\tau)})^2 d\tau,
\end{equation}
where in the last passage we used (\ref{D'}). Combining (\ref{step1}) and (\ref{step2}), we get

\begin{equation}\label{step3}
\int_K |\nabla u|^2 d\gamma=\int_0^{t_0} \left((D^{-1})'(\tau)\right)^2\gamma(K_{D^{-1}(\tau)})^2 d\tau.
\end{equation}

\medskip

Next, for any $\tau\in\R,$ in view of the definition of $u^{\dagger},$ we have
\begin{equation}\label{previous}
	\int\limits_{\{x_1=T^{-1}(\tau)\}} |\nabla u^{\dagger}| d\gamma_{\partial H_{T^{-1}(\tau)}}=\frac{1}{\sqrt{2\pi}} \,e^{-\frac{T^{-1}(\tau)^2}{2}} f'(\tau)T'(T^{-1}(\tau))=-f'(\tau)\gamma(H_{T^{-1}(\tau)})^2,
	\end{equation}
where in the last passage we used the fact that, by Lemma \ref{torhs}, $T'(s)=-\sqrt{2\pi}e^{\frac{s^2}{2}} \gamma(H_s)^2$.

Thus we write, by the change of variables $s=f(\tau),$ and again by the co-area formula:
\begin{equation}\label{finalstep}
\int_{K^\dagger} |\nabla u^{\dagger}|^2 d\gamma= \int_{0}^{u_{\max}^\dagger} \int_{\{u^{\dagger}=s\}} |\nabla u^{\dagger}| d\gamma_{\partial\{u^{\dagger}=s\}} ds=
\end{equation}
$$- \int_{0}^{t_0} f'(\tau)\int_{\{u^{\dagger}=f(\tau)\}} |\nabla u^{\dagger}| d\gamma_{\partial\{u^{\dagger}=f(\tau)\}} d\tau= \int_{0}^{t_0} f'(\tau)^2 \gamma(H_{T^{-1}(\tau)})^2  d\tau,$$
where in the last passage we used (\ref{previous}) together with the fact that $u^{\dagger}(x)=f(\tau)$ whenever $x_1=T^{-1}(\tau).$ Finally, by (\ref{step3}), (\ref{finalstep}) and (\ref{f-def}), we conclude part (1).

\medskip
\medskip

\textbf{Part (2).} We write the layer-cake representation and apply the change of variables $\tau=D(t)$ and $\tau=f^{-1}(t)$:
$$\int_K u \,d\gamma=\int_0^{u_{max}} \gamma(\{u\geq t\}) \,dt=-\int_0^{t_0} \gamma(K_{D^{-1}(\tau)}) (D^{-1})'(\tau) \,d\tau=$$$$ -\int_0^{t_0} \gamma(H_{T^{-1}(\tau)}) f'(\tau) \,d\tau= \int_{K^{\dagger}} u^{\dagger} \,d\gamma,$$
where in the last passage we used the definition of $f$ and $u^{\dagger}$ (\ref{f-def}).

\medskip
\medskip

\textbf{Part (3).} Similarly to part (2), we use the layer cake formula together with the additional change of variables $s=F^{-1}(t)$ and write 

$$\int\limits_K F(u) \,d\gamma=\int\limits_{F(0)}^{F(u_{max})} \gamma(\{F(u)\geq t\}) \,dt=-\int\limits_0^{t_0} F'(D^{-1}(\tau)) \gamma(K_{D^{-1}(\tau)}) (D^{-1})'(\tau) \,d\tau$$
and 
$$\int\limits_{K^{\dagger}} F(u^{\dagger})\, d\gamma=\int\limits_{F(0)}^{F(u^\dagger_{max})} \gamma(\{F(u^{\dagger})\geq t\}) dt=-\int\limits_0^{t_0} F'(f(\tau)) \gamma(K_{T^{-1}(\tau)}) f'(\tau) d\tau.$$
In view of the fact that $F$ is convex, and therefore $F'$ is non-decreasing, and in view of (\ref{f-def}), in order to conclude part (3), we are only left to show that for all $\tau\in [0,t_0],$
\begin{equation}\label{left}
	D^{-1}(\tau)\leq f(\tau).
\end{equation}
Note that by (\ref{f-def}) and (\ref{levset-meas}), in view of the fact that both $D$ and $f$ are non-increasing, we have 
 \begin{equation*}
	-(D^{-1})'(\tau)\leq -f'(\tau).
\end{equation*}
We conclude that
$$D^{-1}(\tau)=-\int_{\tau}^{t_0} (D^{-1})'(\tau) d\tau \leq -\int_{\tau}^{t_0} f'(\tau) d\tau = f(\tau),$$
and (\ref{left}) follows.
\end{proof}

\begin{remark} Note that there is not much room in this construction, and one really has to be very precise when constructing the symmetrization. Indeed, if, say, $\int_K |\nabla u|^2 d\gamma>\int_{K^{\dagger}} |\nabla u^{\dagger}|^2 d\gamma$ (rather than equals), but $\int_K u d\gamma=\int_{K^{\dagger}} u^{\dagger} d\gamma,$ then $T_{\gamma}^{mod}(K;u)< T_{\gamma}(K^{\dagger})$, in view of the definition of the modified torsional rigidity (\ref{modTR}) as well as Proposition \ref{prop1-3}.
\end{remark}

As a corollary, using $F(t)=t^2,$ we immediately get

\begin{corollary}\label{reilly} For any non-negative $u\in W_0^{1,2}(K),$ the Rayleigh quotient does not increase under the $\dagger-$symmetrization. In other words,
$$\frac{\int_K |\nabla u|^2 d\gamma}{\int_K u^2 d\gamma}\geq \frac{\int_{K^{\dagger}} |\nabla u^{\dagger}|^2 d\gamma}{\int_{K^{\dagger}} (u^{\dagger})^2 d\gamma}.$$
\end{corollary}

This immediately implies

\begin{corollary}\label{Lambda-decr} For any  
domain $K$, taking $K^{\dagger}$ to be the $\dagger$-rearrangement of $K$ with respect to the first Dirichlet eigenfunction of $L$, we have $\Lambda_{\gamma}(K)\geq \Lambda_{\gamma}(K^{\dagger})$.
\end{corollary}
\begin{proof} We use the Corollary \ref{reilly} with the function $u$ which minimizes the Rayleigh quotient (i.e. with the first Dirichlet eigenfunction of $L$). Note that it is indeed non-negative: considering $|u|$ instead of $u$ decreases the Rayleigh quotient. We see that
$$\Lambda_{\gamma}(K)=\frac{\int_K |\nabla u|^2 d\gamma}{\int_K u^2 d\gamma}\geq \frac{\int_{K^{\dagger}} |\nabla u^{\dagger}|^2 d\gamma}{\int_{K^{\dagger}} (u^{\dagger})^2 d\gamma}\geq \Lambda_{\gamma}(K^{\dagger}).$$
\end{proof}

Finally, we are ready to outline the

\textbf{Proof of Theorem \ref{mainThm}.} By the definition of the $\dagger-$rearrangement and the modified torsional rigidity, $T_{\gamma}(K)\geq T_{\gamma}(K;u)=T_{\gamma}(K^{\dagger}),$ where $u$ is taken to be the first Dirichlet eigenfunction of $L$ on $K.$ Next, by Corollary \ref{Lambda-decr}, $\Lambda_{\gamma}(K)\geq \Lambda_{\gamma}(K^{\dagger}).$ 

Therefore, we have shown that for any open domain $K$ there exists a half-space $K^{\dagger}$ such that simultaneously, $\Lambda_{\gamma}(K)\geq \Lambda_{\gamma}(K^{\dagger})$ and $T_{\gamma}(K)\geq T_{\gamma}(K^{\dagger})$. 

Consider $H$ to be the half-space such that $T_{\gamma}(K)= T_{\gamma}(H)$. Then, by Lemma \ref{T-mon}, $K^{\dagger}\subset H$, and thus by Lemma \ref{L-mon}, $\Lambda_{\gamma}(K^{\dagger})\geq \Lambda_{\gamma}(H)$. Hence $\Lambda_{\gamma}(K)\geq \Lambda_{\gamma}(H)$. This finishes the proof. $\square$

\begin{remark} Note that a more general result also follows from Theorem \ref{mainthm-sym}: given a convex function $F:\R^+\rightarrow\R,$ define $\Lambda_{\gamma}^F(K)$ to be the minimal value of the following energy:
$$\Lambda_{\gamma}^F(K)=\inf_{u\in W_0^{1,2}(K)}\frac{\int_K |\nabla u|^2d\gamma}{\int_K F(u) d\gamma}.$$
For a domain $K$, let $\tilde{K}$ be the half-space with the same Gaussian torsional rigidity as $K$. Then $\Lambda^F_{\gamma}(K)\geq \Lambda^F_{\gamma}(\tilde{K}).$

\end{remark}

\end{document}